\documentclass[12pt]{article}

\usepackage{amsmath,amssymb,amsthm,url}

\newcommand{\rack}{\vartriangleright}
\newcommand{\frack}{f_{\mathrm{rack}}}
\newcommand{\fquand}{f_{\mathrm{quandle}}}
\newcommand{\fkei}{f_{\mathrm{kei}}}

\newcommand{\im}{\mathrm{im}}
\newcommand{\Sym}{\mathrm{Sym}}

\theoremstyle{plain}
\newtheorem{theorem}{Theorem}
\newtheorem{lemma}[theorem]{Lemma}
\newtheorem{corollary}[theorem]{Corollary}

\theoremstyle{definition}
\newtheorem*{definition}{Definition}

\begin{document}

\title{Enumerating finite racks, quandles and kei}
\author{Simon R. Blackburn\\
Department of Mathematics\\
Royal Holloway, University of London\\
Egham, Surrey TW20 0EX, United Kingdom}
\maketitle

\begin{abstract}
  A rack of order $n$ is a binary operation $\rack$ on a set $X$ of
  cardinality~$n$, such that right multiplication is an
  automorphism. More precisely, $(X,\rack)$ is a rack provided that
  the map $x\mapsto x\rack y$ is a bijection for all $y\in X$, and
  $(x\rack y)\rack z=(x\rack z)\rack (y\rack z)$ for all $x,y,z\in
  X$.

  The paper provides upper and lower bounds of the form $2^{cn^2}$ on
  the number of isomorphism classes of racks of order $n$. Similar
  results on the number of isomorphism classes of quandles and kei are
  obtained. The results of the paper are established by first showing
  how an arbitrary rack is related to its operator group (the
  permutation group on $X$ generated by the maps $x\mapsto x\rack y$
  for $y\in Y$), and then applying some of the theory of
  permutation groups. The relationship between a rack and its operator
  group extends results of Joyce and of Ryder; this relationship might
  be of independent interest.
\end{abstract}

\section{Introduction}
\label{sec:introduction}

We begin by defining the objects of interest to us.

\begin{definition}
A \emph{rack} is a set $X$ together with a binary operator $\rack:
X\times X\rightarrow X$ such that the following two conditions hold.
\begin{itemize}
\item[(i)] For all $y\in X$, the map $f_y:X\rightarrow X$ is a bijection, where we define $xf_y:=(x\rack y)$ for all $x\in X$.
\item[(ii)] For all $x,y,z\in X$, $(x\rack y)\rack z=(x\rack z)\rack (y\rack z)$.
\end{itemize}
\end{definition}

\begin{definition}
A \emph{quandle} is a rack such that $xf_x=x$ for all $x,y\in X$. A \emph{kei} is a quandle such that $f_y$ has order $2$ for all $y\in X$.
\end{definition}

\begin{definition} An isomorphism $\theta$ from a rack $(X,\rack)$ to a rack $(X',\rack')$ is a bijection from $X$ to $X'$ such that $x\theta\rack'y\theta=(x\rack y)\theta$ for all $x,y\in X$.
\end{definition}

Some illustrative examples of kei, quandles and racks are as follows.
If $X$ is a set and $\pi$ is permutation in $\Sym(X)$, the symmetric
group on $X$, then defining $x\rack y:=x\pi$ we have a rack
$(X,\rack)$. If $G$ is a group, then defining $X=G$ and $x\rack
y:=y^{-1}xy$, we find that $(X,\rack)$ is a quandle (a
\emph{conjugation quandle}); if instead we take $X$ to be the set of
all elements of order $2$ in $G$, then $(X,\rack)$ is a kei.

A rack can be motivated purely combinatorially, as a binary operation
where right multiplication is an automorphism. But another motivation
comes from knot theory: kei, quandles and racks have recently led to
the discovery of new invariants of classical knots, and new classes of
generalised knots; see the recent inspiring article by
Nelson~\cite{Nelson11} (which contains a nice exposition of how the
kei, quandle and rack axioms relate to Reidemeister moves in knot
diagrams, as well as mentioning connections with many other areas of
mathematics, and giving more examples of racks). Kei were first
studied by Takasaki~\cite{Takasaki43} in 1943; racks originated in
unpublished correspondence between J.H.\ Conway and G.C.\ Wraith in
1959 on conjugation in groups; the special case of a quandle was
studied in detail from the perspective of knot theory by
Joyce~\cite{Joyce82}. See Fenn and Rourke~\cite{FennRourke92} for a
brief history of these concepts.

The operator group of a rack is the subgroup of $\Sym(X)$ generated by
the permutations $f_y$ for $y\in X$ (see
Section~\ref{sec:rack_structure}). We will establish tight results
that show how a rack can be built from its operator group. These
results extend those of Joyce~\cite[Section~7]{Joyce82} and of
Ryder~\cite[Section~5]{Ryder93}. As an application of these structural
results we prove an enumeration theorem
(Theorem~\ref{thm:main_enumeration}) below, though we hope that the
results are of more general interest.

If $X$ is a finite set of order $n$, we say that a rack, quandle or
kei with underlying set $X$ has order $n$. We write $\frack(n)$,
$\fquand(n)$ and $\fkei(n)$ for the number of isomorphism classes of
racks, quandles and kei of order $n$ respectively. We aim to prove the
following theorem:

\begin{theorem}
\label{thm:main_enumeration}
There exist constants $c_1$ and $c_2$ such that
\[
2^{c_1n^2}\leq \fkei(n)\leq \fquand(n) \leq \frack(n)\leq 2^{c_2n^2}
\]
for all sufficiently large integers $n$.
\end{theorem}
Theorem~\ref{thm:main_enumeration} follows from
Theorems~\ref{thm:lower_bound} and~\ref{thm:upper_bound} below. The
proofs of these theorems show that we may take
$c_1=\frac{1}{4}-\epsilon$ for any positive $\epsilon$, and we may
take $c_2=c+\epsilon$ for any positive $\epsilon$, where $c=\frac{1}{6}(\log_2
24)+\frac{1}{2}(\log_2 3)\approx 1.5566$.

We remark that Theorem~\ref{thm:main_enumeration} shows that the
number of isomorphism classes of kei, quandles and racks grows much
faster than the number of isomorphism classes of groups of order $n$
(which Pyber~\cite{Pyber93} proved is at most $2^{O((\log n)^3)}$; see~\cite{BNV}). In
particular, most quandles are not isomorphic to conjugation
quandles.

The number of kei, quandles and racks grows significantly more slowly
than $n^{n^2}$, the number of binary operations on a set of
cardinality $n$. This contrasts with the situation for semigroups, for
example: Kleitman, Rothschild and Spencer~\cite{KleitmanRothschild76}
have shown that the number of semigroups of order $n$ is
$n^{(1-o(1))n^2}$.

We are not aware of any previous asymptotic enumeration results for
racks, quandles and kei, but there has been interest in enumerating
the quandles of small order. In particular, Ho and Nelson~\cite{HoNelson05}, and
Henderson, Macedo and Nelson~\cite{HendersonMacedo06} have
enumerated the isomorphism classes of quandles of order $8$ or less;
Vendramin~\cite{Vendramin12}, extending computations of
Clauwens~\cite{Clauwens10}, has enumerated the isomorphism classes of
quandles of order 35 or less whose operator group is transitive.

The structure of the remainder of the paper is as follows. In
Section~\ref{sec:rack_structure} we establish the structural results
that relate the structure of a rack with its operator
group. We prove a lower bound
(Theorem~\ref{thm:lower_bound}) on $\fkei(n)$ in
Section~\ref{sec:lower_bound}, and an upper bound
(Theorem~\ref{thm:upper_bound}) on $\frack(n)$ in
Section~\ref{sec:upper_bound}.

\section{The structure of a rack}
\label{sec:rack_structure}

In this section, we recap some terminology we need from the theory of
racks, and prove two structural theorems. These theorems are used to
prove Theorem~\ref{thm:lower_bound} in Section~\ref{sec:lower_bound},
and Theorem~\ref{thm:upper_bound} in
Section~\ref{sec:upper_bound}. Racks in this section need not be
finite.

\begin{definition}
  Let $(X,\rack)$ be a rack. The \emph{augmentation map} of $(X,\rack)$ is the
  map $f:X\rightarrow \Sym(X)$ defined by $(y)f=f_y$. The \emph{operator
    group} (or \emph{inner automorphism group}) of $X$ is the subgroup $G\leq\Sym(X)$ generated by the
  image of $f$. So
\[
G=\langle (y)f:y\in X\rangle.
\]
\end{definition}

Note that, in contrast to some of the literature, we define the
operator group as a permutation group on $X$, rather than as an
abstract group.  The following lemma is well-known; see the second
form of the rack identity in~\cite{FennRourke92}.

\begin{lemma}
\label{lem:augmentation}
Let $(X,\rack)$ be a rack with operator group $G$, and let
$f:X\rightarrow G$ be the augmentation map of $X$. Then
\[
(\alpha g)f = g^{-1}(\alpha f)g
\]
for all $g\in G$ and $\alpha\in X$.
\end{lemma}
\begin{proof}
The identity (ii) in the definition of a rack may be rewritten as
\[
x(yf)(zf)=x(zf)((y(zf))f)
\]
for all $x,y,z\in X$. Set $g=zf$ and set $x'=xg$. Then the
equation above becomes
\[
x'g^{-1}(yf)g=x'((yg)f)
\]
for all $x',y\in X$ and all $g\in \im f\subseteq G$, and so
\begin{equation}
\label{eq:aug}
g^{-1}(yf)g=(yg)f
\end{equation}
for all $x',y\in X$ and all $g\in \im f$. Setting $y=\alpha$, we see
that the lemma holds whenever $g\in\im f$. Setting $y=\alpha g^{-1}$
and multiplying both sides of \eqref{eq:aug} on the left by $g$ and on
the right by $g^{-1}$, we see that the lemma holds when $g^{-1}\in \im
f$. Since the image of $f$ generates $G$, any element of $G$ may be
written as a product as elements from the image of $f$ and their
inverses; the lemma now holds for any $g\in G$, by induction on the
length of such a product.
\end{proof}

Theorems~\ref{thm:rack_construction} and~\ref{thm:rack_general} below
show how to build a rack from its operator group. The theorems
strengthen those of Joyce~\cite[Section~7]{Joyce82}  (who worked with quandles
rather than racks) and Ryder~\cite[Section~5]{Ryder93}.

For a group $G$ and an element $\pi\in G$, we write $C_G(\pi)$ for the
centraliser of $\pi$ in $G$. If $G\leq \Sym(X)$ and $\alpha\in X$, we
write $G_\alpha$ for the point stabiliser of $\alpha$ in $G$.

\begin{theorem}
\label{thm:rack_construction}
Let $X$ be a set, and let $G$ be a subgroup
of $\Sym(X)$. Let $I$ be an index set for the set of orbits of $G$,
and let $\{\alpha_i: i\in I\}\subseteq X$
be a complete set of representatives for the orbits of $G$. For each
$i\in I$ let $\pi_i\in G$, and suppose that
\begin{equation}
\label{eqn:condition_a}
C_G(\pi_i)\geq
  G_{\alpha_i} \text{ for all } i\in I.
\end{equation}
Let $f:X\rightarrow \Sym(X)$ be defined by
\[
(\alpha_ig)f=g^{-1}\pi_ig \text{ for $g\in
G$ and $i\in I$}.
\]
Define $x \rack y = x(yf)$ for all $x,y\in X$. Then $(X,\rack)$ is a rack
with operator group contained in $G$. If
\begin{equation}
\label{eqn:condition_b}
G=\langle g^{-1}\pi_ig:g\in G,
  i\in I\rangle.
\end{equation}
then the operator group of $(X,\rack)$ is equal to $G$.
\end{theorem}
\begin{proof}
Suppose that $(X,\rack)$ is constructed in this way. Note that the map
$f$ is well defined, since~\eqref{eqn:condition_a} is satisfied. Let $x,y,z\in
X$. Let $i,j\in I$ and $g,h\in G$ be such that
$y=\alpha_ig$ and $z=\alpha_jh$. Then
\begin{align*}
(x\rack y)\rack z &=xg^{-1}\pi_igh^{-1}\pi_jh\\
&=x(h^{-1}\pi_jh)(h^{-1}\pi_jh)^{-1}g^{-1}\pi_igh^{-1}\pi_jh\\
&=(xh^{-1}\pi_jh)(gh^{-1}\pi_jh)^{-1}\pi_i(gh^{-1}\pi_jh)\\
&=(x\rack z)\rack(\alpha_igh^{-1}\pi_jh)\\
&=(x\rack z)\rack(y\rack z).
\end{align*}
The maps $x\mapsto x\rack y$ are all permutations, since
$(X)f\subseteq G$. Hence $(X,\rack)$ is a rack. Clearly $f$ is the
augmentation map for $(X,\rack)$. Moreover, we may write the operator group of $(X,\rack)$ as:
\begin{align*}
\langle (\alpha)f:\alpha\in X\rangle
&=\langle (\alpha_ig)f:i\in I,g\in G\rangle\\
&=\langle g^{-1}\pi_ig:i\in I,g\in G\rangle\\
&\leq G,
\end{align*}
with equality in the last line if~\eqref{eqn:condition_b}
holds.
\end{proof}
\begin{theorem}
\label{thm:rack_general}
Every rack on $X$ with operator group $G$ arises in the manner of
Theorem~\ref{thm:rack_construction}. More precisely, let $(X,\rack)$
be a rack with operator group $G$, and let $f$ be the augmentation map
for $(X,\rack)$. Let $\{\alpha_i:i\in I\}$ be a complete set of orbit
representatives for $G$, and define $\pi_i=\alpha_if$ for $i\in I$. 
Then~\eqref{eqn:condition_a} and~\eqref{eqn:condition_b}
hold. Moreover, $(\alpha_ig)f=g^{-1}\pi_ig$
for all $g\in G$ and $i\in I$, and $x\rack y = x(yf)$ for all $x,y\in X$.
\end{theorem}
\begin{proof}
  The last sentence of the theorem follows from
  Lemma~\ref{lem:augmentation}, and the definition of the
  augmentation map. If $g\in G_{\alpha_i}$ we have that $(\alpha_i
  g)f=(\alpha_i)f=\pi_i$, and so $g\in C_G(\pi_i)$ by
  Lemma~\ref{lem:augmentation}. Thus~\eqref{eqn:condition_a}
  holds. Finally, as $G$ is the operator group of $(X,\rack)$,
\begin{align*}
G&=\langle (\alpha)f:\alpha\in X\rangle\\
&=\langle (\alpha_ig)f:i\in I,g\in G\rangle\\
&=\langle g^{-1}(\alpha_if)g:i\in I,g\in G\rangle,
\end{align*}
by Lemma~\ref{lem:augmentation} and so~\eqref{eqn:condition_b} holds.
\end{proof}

We remark that `rack' may be replaced by `quandle' in the theorems
above, provided that we also add the condition that $\pi_i\in
G_{\alpha_i}$ for $i\in I$. `Quandle' may in turn be replaced by `kei' if we
insist in addition that $\pi_i$ has order dividing $2$ for $i\in I$.

As an aside, we end this section by giving two simple consequences of
the Theorems~\ref{thm:rack_construction}
and~\ref{thm:rack_general}. The first is a strengthening of a result
of Ryder~\cite[Theorem~3.2]{Ryder93}, which asserts that every
abstract group is the operator group of some rack.

\begin{corollary}
Every abstract group is the operator group of some quandle. An
abstract group is the operator group of some kei if and only if it is
generated by its involutions.
\end{corollary}
\begin{proof}
  Let $G$ be an abstract group, and let $\{\pi_i\in G:i\in I\}$ be a
  set of elements whose normal closure in $G$ is equal to $G$. Without
  loss of generality, we may assume that there exists $i_0\in I$ such
  that $\pi_{i_0}=1$. For $i\in I\setminus \{i_0\}$, let $X_i$ be the
  set of right cosets of $C_G(\pi_i)$ in $G$, and let
  $\alpha_i=C_G(\pi_i)\in X_i$. Define $X_{i_0}=G$, and
  $\alpha_{i_0}=1\in X_{i_0}$. Let $X$ be the disjoint union of the
  sets $X_i$ (for $i\in I$), and let $G$ act on $X$ by right
  multiplication. Note that $G$ acts faithfully on $X_{i_0}$, so we
  have realised $G$ as a subgroup of $\Sym(X)$.
  Theorem~\ref{thm:rack_construction} shows that $G$ is the operator
  group of a rack. Moreover, using the fact that
  $\alpha_i\pi_i=\alpha_i$ for $i\in I$, it is easy to check that this
  rack is in fact a quandle. This establishes the first statement of
  the corollary.

  The operator group of a kei is generated by its involutions, since
  the definitions of kei and operator group provide a generating set
  consisting of elements of order dividing $2$. Let $G$ be an abstract
  group generated by its involutions. If we define $\{\pi_i\in G:i\in
  I\}$ (for some suitable index set $I$) to be the set of all elements
  of order dividing $2$ in $G$, then the construction above realises
  $G$ as the operator group of a kei.
\end{proof}

We remark that there are many groups that are not generated by their
involutions, and so do not occur as the operator group of any kei. The most
obvious examples of such groups are the non-trivial groups of odd order; more
generally, any group whose Sylow $2$-subgroup is normal and proper is not
generated by its involutions.

\begin{corollary}
Suppose $G\leq\Sym(X)$ is transitive. If $G$ is the
operator group of a rack $(X,\rack)$, then there exists $\pi\in G$
whose normal closure is equal to $G$. Thus not all permutation groups occur as operator groups.
\end{corollary}
\begin{proof}
Suppose $(X,\rack)$ is a rack with operator group $G$. Then
Theorem~\ref{thm:rack_general} implies that there exists $\pi\in G$
whose normal closure is equal to $G$ and so the first statement of the
corollary follows.

Let $G$ be a non-cyclic abelian group acting transitively on $X$.
Suppose, for a contradiction, that $G$ is the operator group of a
rack $(X,\rack)$. Let $\pi\in G$ be an element whose normal closure in $G$ is
equal to $G$. Then, since $G$ is abelian, $G=\langle \pi\rangle$ and
so $G$ is cyclic. This contradiction establishes the final assertion
of the corollary.
\end{proof}

\section{A lower bound}
\label{sec:lower_bound}

\begin{theorem}
\label{thm:lower_bound}
The number $\fkei(n)$ of isomorphism classes of kei of order $n$ is at least $2^{\frac{1}{4}n^2-O(n\log n)}$.
\end{theorem}
We remark that Theorem~\ref{thm:lower_bound} establishes the lower bound in Theorem~\ref{thm:main_enumeration}.

\begin{proof}[Proof of Theorem~\ref{thm:lower_bound}]
Let $X=\{1,2,\ldots,n\}$, and define $k=\lfloor n/2\rfloor$. Set
$T=\{1,2,\ldots,k\}$.

Let $E=(e_{ij})$ be a $k\times k$ matrix such that $e_{ij}\in\{0,1\}$
for all $i,j\in T$ with $i\not=j$, and such that $e_{ii}=0$.  To prove
the theorem, it suffices to construct kei $X_E$ on $X$, in such a way
that $X_E\not=X_{E'}$ whenever $E\not=E'$. (It will happen that
$X_E\cong X_{E'}$ in some cases.) To see that the theorem follows from
this, first note that there are $2^{(k-1)k}$ matrices $E$. Moreover,
an isomorphism class of kei sharing the same underlying set $X$ can
contain at most $n!$ distinct elements, since there are at most $n!$
choices for an isomorphism $\theta: X\rightarrow X$. So we will have
constructed at least $2^{(k-1)k}/n!$ distinct isomorphism classes of
kei. Since $n!\leq n^n=2^{n\log n}$ and $k\geq (n-1)/2$ the theorem
will therefore follow.

For $i\in T$, let $X_i=\{2i-1,2i\}$ and
define $X_{k+1}=\{n\}$. Define $I=T$ when
$n$ is even, and $I=T\cup\{k+1\}$ when $n$ is odd. So
$\bigcup_{i\in I} X_i$ is a partition of $X$ containing exactly $k$
subsets of size $2$, and possibly a single set of size~$1$.

Let
\[
G=\{\pi\in \Sym(X):X_i\pi=X_i\text{
  for all }i\in I\}\leq \Sym(X).
\]
 Then $G$ is an elementary abelian $2$-group
of order $2^k$, generated by the transpositions $\tau_i=(2i-1,2i)$ for
$i\in T$. The orbits of $G$ are the sets $X_i$ where $i\in I$. Define
$\alpha_i=2i-1$ for $i\in I$. Then $\{\alpha_i:i\in I\}$ is a complete set of
representatives for the orbits of $G$.

\begin{figure}
\[E=\begin{pmatrix}
0&1&0\\
1&0&1\\
1&0&0
\end{pmatrix}
\begin{array}{c|cccccccc}
\rack&1&2&3&4&5&6&7\\\hline
1&1&1&2&2&1&1&1\\
2&2&2&1&1&2&2&2\\
3&4&4&3&3&4&4&3\\
4&3&3&4&4&3&3&4\\
5&6&6&5&5&5&5&5\\
6&5&5&6&6&6&6&6\\
7&7&7&7&7&7&7&7
\end{array}
\]
\caption{An example of a kei $X_E$ when $n=7$}
\label{fig:kei_example}
\end{figure}
We construct each kei $X_E$ as follows (see
Figure~\ref{fig:kei_example} for an example). Define permutations
$\pi_i\in G$ for $i\in I$ as follows. For $i\in T$, define
\[
\pi_i=\tau_{1}^{e_{1i}}\tau_{2}^{e_{2i}}\cdots \tau_k^{e_{ki}}.
\]
For $i\in I\setminus T$ (so $n$ is odd and $i=k+1$), let $\pi_i$ be
the identity permutation. Note that distinct matrices $E$ give rise to
distinct lists of permutations $(\pi_i:i\in I)$. The
condition~\eqref{eqn:condition_a} of Theorem~\ref{thm:rack_construction}
is satisfied since $G$ is abelian and $\pi_i\in G$ for $i\in I$, so we may define $f:X\rightarrow G$
and $(X,\rack)$ as in Theorem~\ref{thm:rack_construction}. Let
$X_E=(X,\rack)$. By Theorem~\ref{thm:rack_construction}, $X_E$ is a
rack whose operator group is contained in $G$. It is not hard to check
that $x\rack x=x$, using the fact that $e_{ii}=0$ for all $i\in T$,
and so $X_E$ is a quandle. All the elements of the operator group of
$X_E$ have order dividing $2$, since the operator group is contained
in the elementary abelian $2$-group $G$. Thus, $X_E$ is a
kei. Finally, since $\pi_i$ is equal to the map $x\mapsto x\rack
\alpha_i$, distinct matrices $E$ give rise to distinct kei
$(X,\rack)$. So the theorem follows.
\end{proof}

\section{An upper bound}
\label{sec:upper_bound}

This section aims to prove the following theorem.

\begin{theorem}
\label{thm:upper_bound}
The number $\frack(n)$ of isomorphism classes of racks of order $n$ is
at most $2^{(c+o(1))n^2}$, where $c=\frac{1}{6}(\log_2
24)+\frac{1}{2}(\log_2 3)\approx 1.5566$.
\end{theorem}

We remark that the proof of this theorem will complete the proof of
Theorem~\ref{thm:main_enumeration}. We require the following two
results from the theory of permutation groups. The following theorem
is due to Laci Pyber~\cite[Corollary~3.3]{Pyber93}.

\begin{theorem}
\label{thm:perm_group_bound}
The number of subgroups of $\Sym(X)$ with $|X|=n$ is bounded
above by $24^{(\frac{1}{6}+o(1))n^2}$.
\end{theorem}

The next theorem is due to Attila Mar\'oti~\cite{Maroti05},
extending work of Kov\'acs and Robinson~\cite{KovacsRobinson93}, and
of Riese and Schmid~\cite{RieseSchmid03}.

\begin{theorem}
\label{thm:conj_class_bound}
Let $n>2$. Let $G$ be a subgroup of $\Sym(X)$, with $|X|=n$. Then the
number of conjugacy classes of $G$ is bounded above by $3^{(n-1)/2}$.
\end{theorem}

\begin{proof}[Proof of Theorem~\ref{thm:upper_bound}]
Let $X$ be a set with $|X|=n$. By Theorem~\ref{thm:perm_group_bound},
there are at most $24^{(\frac{1}{6}+o(1))n^2}$ subgroups $G$ of $\Sym(X)$, and so
the theorem will follow if we can provide a sufficiently good upper bound on
the number of racks with operator group $G$, for any fixed $G$.

Let $G$ be a fixed subgroup of $\Sym(X)$. Suppose $G$ has $s$
orbits, of lengths $n_1,n_2,\ldots ,n_s$. Clearly $s\leq n$, and
$\sum_{i=1}^s n_i=n$.

Let $\alpha_1,\alpha_2,\ldots,\alpha_s$ be a complete set of
representatives for the orbits of~$G$. By
Theorem~\ref{thm:rack_general}, a rack with operator group $G$ is
determined by a sequence of elements $\pi_1,\pi_2,\ldots,\pi_s\in G$ such that
$C_G(\pi_i)\geq G_{\alpha_i}$. Since $G_{\alpha_i}$ has index
$n_i$ in $G$, each $\pi$ lies in a $G$-conjugacy class $\Pi_i$
of order at most $n_i$.

By Theorem~\ref{thm:conj_class_bound}, the group $G$ has at most
$3^{\frac{1}{2}n}$ conjugacy classes. There are at most $3^{\frac{1}{2}ns}$ choices
for the conjugacy classes $\Pi_1,\Pi_2,\ldots ,\Pi_s$, and $3^{\frac{1}{2}ns}\leq
3^{\frac{1}{2}n^2}$. Once these
conjugacy classes are fixed, there are at most $n_i$ choices for each
element $\pi_i\in G$. So the number of choices for the elements
$\pi_i$ once the classes $\Pi_i$ are chosen is at most
$\prod_{i=1}^sn_i$. The product $\prod_{i=1}^tm_i$ of
positive integers $m_i$ such that $\sum_{i=1}^tm_i=n$ is maximised
when $m_i\leq 3$ for all $i$, since $(m-2)m\geq m$ when $m\geq
4$. So $\prod_{i=1}^sn_i\leq 3^n=2^{O(n)}$.

Thus there are at most $2^{(c+o(1))n^2)}$ racks on $X$. Since every rack of
order~$n$ is isomorphic to a rack with underlying set $X$, there are at
most $2^{(c+o(1))n^2)}$ isomorphism classes of racks of order $n$, as required.
\end{proof}

\paragraph{Acknowledgement} The author would like to thank Colin
Rourke for sending a copy of Hayley Ryder's
thesis~\cite{Ryder93}, and Sam Nelson for providing references for
work on the enumation of small quandles.


\begin{thebibliography}{99}
\bibitem{BNV} Simon R. Blackburn, Peter M. Neumann and Geetha
  Venkataraman, \emph{Enumeration of Finite Groups}, Cambridge
  University Press, Cambridge, UK, 2007.
\bibitem{Clauwens10} F.J.-B.J. Clauwens, `Small connected quandles',
  preprint, \url{http://arxiv.org/abs/1011.2456}.
\bibitem{FennRourke92} Roger Fenn and Colin Rourke, `Racks and links in
  codimension two', \emph{J. Knot Theory Ramifications} \textbf{1}
  (1992), 343--406.
\bibitem{HendersonMacedo06} Richard Henderson, Todd Macedo and
  Sam Nelson, `Symbolic computation with finite quandles',
  \emph{J. Symbolic Comput.} \textbf{41} (2006), 811--817.
\bibitem{HoNelson05} Benita Ho and Sam Nelson, `Matrices and finite
  quandles', \emph{Homology, Homotopy Appl.} \textbf{7}
  (2005), 197--208.
\bibitem{Joyce82} David Joyce, `A classifying invariant of knots, the
  knot quandle', \emph{J. Pure Appl. Algebra} \textbf{23} (1982)
  37--65.
\bibitem{KleitmanRothschild76} Daniel J. Kleitman, Bruce L. Rothschild
  and Joel H. Spencer, `The number of semigroups of order $n$',
  \emph{Proc. Amer. Math. Soc.} \textbf{55} (1976), 227--232. 
\bibitem{KovacsRobinson93} L.G.\  Kov\'acs and Geoffrey R.\  Robinson, `On
  the number of conjugacy classes of a finite group',
  \emph{J. Algebra} {\bf 160} (1993), 441--460.
\bibitem{Maroti05} Attila Mar\'oti, `Bounding the number of conjugacy
  classes of a permutation group', \emph{J. Group Theory} {\bf 8}
  (2005), 273--289.
\bibitem{Nelson11} Sam Nelson, `The combinatorial revolution in knot
  theory', \emph{Notices Amer. Math. Soc.} \textbf{58} (2011),
  1553--1561.
\bibitem{Pyber93} L.\ Pyber, `Enumerating finite groups of a given
  order', {\em Annals of Math.}\ {\bf 137} (1993), 203--220.
\bibitem{RieseSchmid03} Udo Riese and Peter Schmid, `Real vectors for
  linear groups and the $k(GV)$-problem', \emph{J. Algebra}
  \textbf{267} (2003), 725--755.
\bibitem{Ryder93} Hayley Jane Ryder, \emph{The Structure of Racks},
  University of Warwick, PhD Thesis, August 1993.
\bibitem{Takasaki43} M. Takasaki, `Abstractions of symmetric
  functions', \emph{Tohoku Math. J.} \textbf{49} (1943), 143--207.
\bibitem{Vendramin12} L. Vendramin, `On the classification of quandles
  of low order', \emph{J. Knot Theory Ramifications}, to appear.
\end{thebibliography}
\end{document}